%% file: defie_anal.tex
\documentclass[english]{article}

\RequirePackage{fix-cm}

\usepackage{amsmath,amssymb,empheq}
\usepackage{amsthm}
\usepackage[utf8]{inputenc}
\usepackage[english]{babel}
\usepackage[T1]{fontenc}
\usepackage{graphicx,subfig}
\usepackage{cite}

\newcommand\subst[1]{\big|_{#1}}
\newcommand\divgs{\nabla_s\cdot}
\newcommand\grads{\nabla_s}

\newcommand\pdev[1]{\frac{\partial}{\partial #1}}
\newcommand\piola[1]{\mathop{\mathcal{P}_{#1}}}

\newcommand\reals{\mathbb{R}}

\newcommand\supp{\mathrm{supp}\,}

\newcommand\Tq{{T_q}}
\newcommand\Tp{{T_p}}

\newcommand\TRI{\mathcal{T}}
\newcommand\RT{\mathrm{RT}}

\newcommand\hh{\mathbf{H}}
\newcommand\ee{\mathbf{E}}
\newcommand\jcurr{\mathbf{J}}
\newcommand\uu{\mathbf{u}}
\newcommand\vv{\mathbf{v}}
\newcommand\nor{\hat{\mathbf{n}}}
\newcommand\Ttau{\boldsymbol{\tau}}

\newcommand{\xx}{\boldsymbol{\mathrm{x}}}

\newcommand{\aalpha}{\boldsymbol{\alpha}}
\newcommand{\Aa}{\boldsymbol{\mathrm{A}}}
\newcommand{\bb}{\boldsymbol{\mathrm{b}}}

\newtheorem{lem}{Lemma}

\title{An elementary formula for computing shape derivatives of EFIE
  system matrix} 

  \author{Kataja, Juhani\thanks{\small Aalto University,
    Dept. Rad. Sci. and Eng.,
      P.O. Box 13000 FI-00076 AALTO, Finland
      \small\texttt{juhani.kataja@aalto.fi}}\, \thanks{Author was funded by Academy of Finland.}
\and Toivanen, Jukka I. \thanks{\small University of Jyväskylä, PL 35 (Agora), 
               FI-40014 Jyväskylän yliopisto, Finland
               \texttt{jukka.i.toivanen@jyu.fi}}}

\begin{document}
\maketitle
\begin{abstract}
  We derive analytical shape derivative formulas of the system matrix representing
  electric field integral equation discretized with
  Raviart-Thomas basis functions. The arising integrals are easy to compute with similar
  methods as the entries of the original system matrix.
  The results are compared to derivatives
  computed with automatic differentiation technique and finite
  differences, and are found to be in excellent agreement.
\end{abstract}
\section{Introduction}

The adjoint variable methods based on shape gradients have recently
gained significant attention in the shape optimization of microwave
devices \cite{nikolova2006_avm,Toivanen2009ESA}. The main reason for
such an interest is that the adjoint variable approach eases
significantly the computational burden of the optimization process:
computing the gradient of the objective requires only
one solve of the governing state equations, and in addition, at most
one solve of the adjoint problem.

The adjoint variable methods rely on the availability of the
derivatives of the system matrix with respect to the control
parameters. Traditionally these derivatives have been computed with finite
difference (FD) formulas. 
The main drawback of this approach is that the
accuracy greatly depends on the step length parameter, which has to be
carefully chosen in order to find a suitable balance between round off
and truncation errors.

Use of the automatic differentiation (AD) 
was proposed in \cite{Toivanen2009ESA}
to compute the derivatives  
of the system matrix arising from a method of moments 
discretization of the electric field integral equation (EFIE). 
The  method was applied to antenna shape optimization
problems in \cite{Toivanen2009ESA, uwb_splines}.
AD computes the exact derivatives of the computer code,
but there are some complications involved.

AD can be implemented either as a source transformation tool, or using
operator overloading. Source transformation tools are complicated to
implement, and existing tools are available only for some programming
languages. While using these tools, one may be forced to restructure
the code, and avoid such features of the programming language that are
not supported. Tools based on the operator overloading are more simple
to implement, but not all programming languages support operator
overloading. Moreover, there is some execution overhead related to
this approach. In both cases, AD has to be applied also to all
subroutines that are called from the code, which may present
difficulties if one wishes to use any external libraries.

The electric field integral equation is a widely used method to
compute scattering of time harmonic electromagnetic field from
perfectly electrically conducting (PEC) bodies and surfaces
\cite{efie,efie_christiansen,kolundzija-djordjevic}.
In this paper we derive simple analytical formulas for calculating the
system matrix derivatives of the discretized  EFIE system. 
The arising integrals are such that they are
easy to compute using methods similar to those used to calculate the
elements of the original system matrix. The results are compared
against derivatives calculated using automatic differentiation and
difference formulas.

We employ the lowest order Raviart-Thomas (RT) basis functions
\cite{raviart-thomas} in the discretization. They are almost identical
to the Rao-Wilton-Glisson (RWG) basis functions \cite{rwg}, which are
the usual choice to discretize the EFIE, except that usually the RWG
functions are scaled with the length of the edge which they are
related to.

The shape derivatives of the electromagnetic field solution together
with the far field pattern has been characterized and analyzed in
\cite{costabel_shape_deriv,costabel_shape_deriv_II,potthast_dom_deriv}. 
Contrary to these works, we consider differentiation of the
discretized system, and utilize the adjoint 
variable method.

\section{Preliminaries}
Let $S$ be a surface with or without boundary in $\reals^3$ and $\TRI$
its triangulation. It can be a boundary of a some sufficiently regular
open domain in $\reals^3$, in which case it models the scattering
object. If it is not a boundary of any set, but it otherwise exhibits
sufficient regularity, it acts as a model for a thin perfectly
conducting screen.

In what follows, $\RT(\TRI)$ denotes the space spanned by the 
lowest order Raviart-Thomas (RT) basis functions on $\TRI$,
$\divgs$ denotes the surface divergence, and $\grads$
denotes the surface gradient.

In the EFIE the unknown function to be solved is, roughly speaking,
the equivalent surface current $\jcurr=\nor\times\hh$, where $\nor$ is
the surface unit normal and $\hh$ is the total magnetic field. The
source term in the equation is the incident electric field denoted by
$\ee_p$.

The current $\jcurr$ satisfies the Rumsey reaction principle
\cite{kolundzija-djordjevic,buffa_bem,schwab_bem,rumsey54}:

\emph{Find $\jcurr\in H^{-1/2}_\times(\divgs;S)$ s.t.}
\begin{align}
  \nonumber
  \frac{i}{\omega\epsilon} \int_S \divgs\uu(x) & \int_S g_k(x-y) \divgs\jcurr(y) dy dx -
  i\omega\mu \int_S  \uu(x) \cdot \int_S g_k(x-y) \jcurr(y) dy dx \\
  & = \int_S \uu(x) \cdot \ee_p(x) dx\quad \forall \uu\in H^{-1/2}_\times(\divgs;S).  
  \label{equ:rumsey}
\end{align}
Here $g_k$ is the fundamental solution of the Helmholtz equation:
\begin{equation}
g_k(x) = \frac{e^{ik|x|}}{4\pi|x|},
\end{equation}
and $\divgs$ is the surface divergence. We denote the Euclidean norm
of $x\in\reals^3$ by $|x|$.

For detailed treatment of the space $H^{-1/2}_\times(\divgs;S)$ and
the above equation we refer to \cite{buffa_bem}. We
note that the integrals in \eqref{equ:rumsey} should be interpreted as
a certain duality pairing, but at the discrete level they are proper
integrals.

After discretizing the equation \eqref{equ:rumsey} we arrive to the
following finite dimensional variational problem:

\emph{Find $\jcurr_h\in\RT$ such that}
\begin{equation}
  a(\vv,\jcurr_h) = \int_S\vv(x)\cdot\ee_p(x) dx,\quad \forall \vv\in\RT(\TRI),
  \label{equ:rumsey_finite}  
\end{equation}
where the bilinear form $a$ is given by
\begin{align}
  a(\vv,\uu) = &\frac{i}{\omega\epsilon} \int_S \divgs\vv(x)
  \int_S g_k(x-y) \divgs\uu(y) dy dx - \nonumber \\
  &i\omega\mu \int_S \vv(x) \cdot \int_S g_k(x-y) \uu(y) dy dx.
\end{align}
We denote the corresponding linear system of equations of
\eqref{equ:rumsey_finite} by
\begin{equation}
  \Aa\xx = \bb.
\end{equation}

Let $T_p,T_q\in\TRI$, $\uu,\vv\in\RT(\TRI)$, and $\supp \uu$ and
$\supp\vv$ intersect $T_p$ and $T_q$, resp. In order to compute the
entries of $\Aa$ one needs to compute integrals of type
\begin{align}
  \label{equ:curr_efie}
  I_1(\uu,\vv;\Tp,\Tq) & = \int_\Tq \vv(x)\cdot \int_\Tp g_k(x-y) \uu(y) dy dx,\\
  \label{equ:charge_efie}
  I_2(\uu,\vv;\Tp,\Tq) & = \int_\Tq \divgs \vv(x) \int_\Tp g_k(x-y) \divgs \uu(y) dy dx.
\end{align}

We denote the usual interpolating nodal piecewise first order
polynomials on $S$ associated with $\TRI$ by $(\lambda_m)_{m=1}^N$,
where $N$ is the number of vertices in $\TRI$.

\subsection{Adjoint variable methods for the EFIE system}
  
Let $\aalpha$ be the vector of design variables, and
$\mathcal{J}(\aalpha, \xx(\aalpha))$ be a real valued objective
function.  In general,
real valued functions are not differentiable with respect to complex
arguments in the conventional complex analytic sense.  Therefore we
differentiate $\mathcal{J}$ separately with respect to the real and
imaginary parts of the variables $\xx$:
\begin{equation}
\label{eq:dJ}
\frac{\rm d \mathcal{J}}{\rm d \alpha_k} 
= \frac{\partial \mathcal{J}}{\partial \alpha_k} 
+ \frac{\partial \mathcal{J}}{\partial \Re \xx} \frac{\partial \Re \xx}{\partial \alpha_k}
+ \frac{\partial \mathcal{J}}{\partial \Im \xx} \frac{\partial \Im \xx}{\partial \alpha_k} .
\end{equation}
Here $\partial \mathcal{J} / \partial \alpha_k$ reflects the explicit dependence of
$\mathcal{J}$ on the design.

The convention
$\nabla_{\xx} \mathcal{J} := \nabla_{\Re \xx} \mathcal{J} + i \nabla_{\Im \xx} \mathcal{J}$
will be used to simplify notation.
It can be easily checked that equation \eqref{eq:dJ} can now be written as
\begin{equation}
\label{eq:dJs}
\frac{\rm d \mathcal{J}}{\rm d \alpha_k} 
= \frac{\partial \mathcal{J}}{\partial \alpha_k} 
+ \Re \left[ ( \nabla_{\xx} \mathcal{J} )^H \frac{\partial \xx}{\partial \alpha_k} \right].
\end{equation}
By differentiating the state relation $\Aa \xx = \bb$, we obtain
\begin{equation}
\label{eq:state_der}
\Aa \frac{\partial \xx}{\partial \alpha_k} = -\frac{\partial \Aa}{\partial \alpha_k} \xx
+  \frac{\partial \bb}{\partial \alpha_k}.
\end{equation}
In the so called direct differentiation approach, the derivatives
$\partial \xx / \partial \alpha_k$ are solved from
\eqref{eq:state_der}, and \eqref{eq:dJs} is used to compute the
gradient of the objective function. However, the right hand side of
\eqref{eq:state_der} is different for each design variable, which
makes this approach rather inefficient.

By introducing the adjoint problem
\begin{equation}
\label{eq:adjoint}
A^H \gamma = \nabla_{\xx} \mathcal{J}, 
\end{equation}
and using \eqref{eq:state_der},
equation \eqref{eq:dJs} can be written as
\begin{align}
\frac{\partial \mathcal{J}}{\partial \alpha_k} &= 
\frac{\partial \mathcal{J}}{\partial \alpha_k}  +
  \Re \left[ (\nabla_{\xx} \mathcal{J})^H \Aa^{-1} \left(\frac{\partial \bb}{\partial \alpha_k} 
 - \frac{\partial \Aa}{\partial \alpha_k} \xx \right)\right]  \\
&= \frac{\partial \mathcal{J}}{\partial \alpha_k} 
 + \Re \left[ (\gamma^H \left(\frac{\partial \bb}{\partial \alpha_k} 
 - \frac{\partial \Aa}{\partial \alpha_k} \xx \right)\right]. \label{eq:adj_sens}
\end{align}
In the adjoint variable method, equation \eqref{eq:adjoint} is solved for $\gamma$, and 
equation \eqref{eq:adj_sens} is used to compute the gradient of the objective.
Notice that the adjoint problem \eqref{eq:adjoint} does not depend on the
design, and therefore $\gamma$ is the same for all design variables.
In case of some particular objective functions (see e.g. \cite{nikolova2006_avm}), 
the adjoint vector $\gamma$ can be
immediately obtained from the solution vector $\xx$, and one does not have to solve
the adjoint problem at all.

In any case, the derivatives of the right hand side and the system
matrix with respect to the design variables are required.
This paper presents analytical formulas, which can be used to compute
them efficiently in the case where $\Aa$ is a system matrix arising
from the EFIE formulation \eqref{equ:rumsey_finite}.

\section{The derivative formulas}
We define the shape deforming mapping with the aid of the global nodal
shape functions $\lambda_m$ as follows.

Let $\Ttau\in \reals^3$ be a vector and $0\leq s<1$. We define the
deformation mapping $F_s^m$ associated with node $m$ by
\begin{equation}
F_s^m(x) = x + s \Ttau \lambda_m(x).
\label{equ:Fs_char}
\end{equation}

Note that the domain of $F_s^m$ is the whole of $S$.  We shall drop the
superscript $m$ from $F_s^m$ whenever it does not cause confusion.

Let $K$ be a triangle in $\reals^2$ and $F$ a diffeomorphism from $K$
to its image. We define the Piola transform $\piola{F}$ associated
with $F$ by
\begin{equation}\piola{F}\uu = \frac{1}{\det F'} F' \uu \circ F^{-1},\end{equation}
which maps smooth sections of the tangent bundle of $K$ to those of
$F(K)$. Here the Jacobian of $F$ by is denoted by $F'$. We note that
the divergence of the Piola transformed function $\uu$ is given by
\begin{equation}\divgs \piola{F}\uu = \frac{1}{\det F'} (\divgs
  \uu)\circ F^{-1}.\end{equation}
This simplifies the calculations greatly, as we shall see.

The Raviart-Thomas basis functions can be defined using the Piola
transform and reference element as follows. Let $\widehat{K}$ be the
standard 2-simplex \[\{ (x,y)\in\reals^2:\, 0<x,\, 0<y,\,
\mathrm{and}\, 0<x+y<1\}\] and define functions $\widehat{\uu}_i$ on
$\widehat{K}$ by
\begin{equation}\begin{cases}
\widehat{\uu}_1(x,y) = [x,\, 1-y]^T,& \\
\widehat{\uu}_2(x,y) = [x,\, y]^T& \mathrm{and}\\
\widehat{\uu}_3(x,y) = [1-x,\, y]^T.
\end{cases}\end{equation}

Now the restriction of an arbitrary $\uu\in\RT(\TRI)$ to the triangle
$K$ is a Piola transform of some $\widehat{\uu}_i$ ie.
\[\uu = \piola{F_K} \widehat{\uu}_i\]
for some $i\in\{1,2,3\}$, where $F_K$ is the affine mapping which maps
$\widehat{K}$ to $K$. We note that the the determinant of $F_K'$ is
given by $\det F_K' = \pm 2 |K|$, where $|K|$ is the area of $K$ and
the sign is determined by its orientation.

The main objects of interest in this work are the derivatives
\begin{align}
\pdev{s} I_m(\piola{F_S}\uu,\piola{F_s}\vv; F_s(\Tp),F_s(\Tq))\subst{s=0},\quad m = 1,2.
\end{align}

To start with, we note that
\begin{align}
  I_1&(\piola{F_s} \uu, \piola{F_s} \vv;  F_s(\Tp), F_s(\Tq)) = \nonumber \\
  &\int_{F_s(\Tq)}\int_{F_s(\Tp)} \frac{ F_s'(x)\vv \circ F_s^{-1}(x)}{\det F_s'(x)} \cdot 
  \frac{F_s'(y)\uu\circ F_s^{-1} (y)}{\det F_s'(y)}  g_k(x-y) dy dx = \nonumber \\
  &\int_\Tq\int_\Tp F_s'(x) \vv(x) \cdot F_s'(y) \uu(y) g_k(F_s(x)-F_s(y)) dy dx.\label{equ:I1big}
\end{align}
and
\begin{align}
I_2&(\piola{F_s} \uu, \piola{F_s} \vv;  F_s(\Tp), F_s(\Tq)) = \nonumber \\
&\int_{F_s(\Tq)}\int_{F_s(\Tp)} \frac{(\divgs \vv) \circ F_s^{-1}(x)}{\det F_s'(x)} \cdot 
  \frac{(\divgs \uu)\circ F_s^{-1} (y) }{\det F_s'(y)} g_k(x-y) dy dx = \nonumber \\
  &\int_\Tq\int_\Tp \divgs \vv(x) \cdot \divgs \uu(y) g_k(F_s(x)-F_s(y)) dy dx.\label{equ:I2big}
\end{align}
For conciseness, we denote $I_m^s = I_m(\piola{F_s} \uu, \piola{F_s}
\vv; F_s(\Tp), F_s(\Tq))$.

Looking at \eqref{equ:I1big} and \eqref{equ:I2big} we find that we
only need to compute derivatives of $g_k$ and $F'_s$ wrt. $s$.
\begin{lem} Let $F_s^m$ be given by Eq. \eqref{equ:Fs_char}.
  It holds that \begin{equation}
    \pdev{s}|F_s^m(x)-F_s^m(y)|\subst{s=0}   = \frac{(x-y)\cdot \Ttau
    (\lambda_m(x)-\lambda_m(y))}{|x-y|}\end{equation}
and
\begin{equation}
  \pdev{s} F_s'(x)^T F_s'(y)\subst{s=0} = (\Ttau \grads \lambda_m(x))^T + \Ttau\grads\lambda_m(y).
\end{equation}

\end{lem}
\begin{proof} Just by calculating
\begin{align*}
  |F_s^m(x)-F_s^m(y)|^2 & = |x-y + s\Ttau\left(\lambda_m(x)-\lambda_m(y)\right)|^2 \\
  & = \, A + 2s (x-y)\cdot \Ttau (\lambda_m(x)-\lambda_m(y)) + B s^2,
\end{align*}
where $A$ and $B$ are constants wrt. $s$. By using chain rule, the first assertion follows.

For the second formula, we note that
\begin{equation}F_s'(x) = I + s\Ttau\grads\lambda_m(x)\end{equation}
and the assertion is obvious.
\end{proof}

Using the above results we have the shape derivative of the kernel
$g_k$ as
\begin{equation}
  \pdev{s} g_k(F_s(x)-F_s(y))\subst{s=0} = 
  g_k(x-y)\left(i k - \frac{1}{|x-y|}\right) \frac{(x-y)\cdot \Ttau
    (\lambda_m(x)-\lambda_m(y))}{|x-y|}.
  \label{equ:kernel_deriv}
\end{equation}
Furthermore, this can be simplified to
\[\pdev{s} g_k(F_s(x)-F_s(y))\subst{s=0} = \Ttau\cdot(\nabla g_k)(x-y)(\lambda_m(x)-\lambda_m(y)).\]
Thus, we immediately get the formula
\begin{equation}\boxed{
  \pdev{s} I_2^s = \int_\Tq\int_\Tp \divgs \vv(x) \divgs \uu(y)
  \Ttau\cdot(\nabla g_k)(x-y) \left(\lambda_m(x)-\lambda_m(y)\right)dy dx.} \label{equ:diffcharge}
\end{equation}

The derivative $\pdev{s} I_1^s$ is only slightly more complicated;
it is obtained by an application of the Leibniz's rule:
\begin{align}
  \pdev{s} I_1^s & = \int_\Tq\int_\Tp  \left(\pdev{s} \vv(x) \cdot F_s'(x)^T F_S'(y) \uu(y)\right) g_k(x-y) + \nonumber\\
& \hspace{1cm}\vv(x) \cdot \uu(y) \pdev{s} g_k(F_s(x)-F_s(y)) dy dx.
\end{align}
Thus we obtain
\begin{empheq}[box=\fbox]{align}
\pdev{s} I_1^s & = \int_\Tq\int_\Tp \vv(x)\cdot\left((\Ttau \grads \lambda_m(x))^T + \Ttau\grads\lambda_m(y)\right)\cdot \uu(y) g_k(x-y)+\nonumber \\
& \hspace{1cm} \vv(x)\cdot\uu(y)  \Ttau\cdot(\nabla g_k)(x-y) \left(\lambda_m(x)-\lambda_m(y)\right)dy dx.\label{equ:diffcurr}
\end{empheq}

\section{Numerical verifications}
We compare the results given by the formulas \eqref{equ:diffcurr} and
\eqref{equ:diffcharge} to the values of derivative computed with
difference formulas and automatic differentiation applied to numerical
code which computes $I_m^s$ at $s=0$.

To start with, we note that the integrals \eqref{equ:diffcurr} and
\eqref{equ:diffcharge} are singular when the closures of $T_p$ and
$T_q$ intersect. Thus we manipulate the expressions in such a way that we
can directly apply singularity subtraction methods to compute
them.

The singularity subtraction method is designed to calculate singular
integrals of the form
\begin{equation}
\int v(x) \int K(x-y) u(y) dy dx,
\end{equation}
where $K$ is singular at $0$. The idea is to express $K$ as a sum of
singular $K_s$ and a function $K_b$ bounded at $0$, in such a way that
$\int K_s(x-y) v(x)dx$ can be evaluated analytically, and $\int
K_b(x-y) v(x) dx$ integrates easily with numerical quadratures. In
this paper, we use the same quadrature to integrate $K = K_s + K_b$
wrt.\ $y$ and the inner integral with kernel $K_b$ wrt.\ $x$. Note
that $\int K_b(x-y) v(x) dx$ is integrated analytically.

Looking at equations \eqref{equ:diffcharge} and \eqref{equ:diffcurr}
we observe that we only need to compute integrals of type
\begin{equation}
\begin{cases}
  \displaystyle\int\divgs\vv(x) \int \Ttau\cdot\nabla_x g_k(x-y)
  \left(\lambda_m(x)-\lambda_m(y)\right)\divgs\uu(y) dy dx& \\
  \displaystyle\int \vv(x)\cdot \int \Ttau\cdot\nabla_x g_k(x-y)\left(\lambda_m(x)-\lambda_m(y)\right)\uu(y) dy dx & \\
  \displaystyle\int\int \vv(x)\cdot T(x,y)\uu(y) g_k(x-y) dy dx
\end{cases}
\end{equation}

Here $T(x,y) = (\Ttau \grads \lambda_m(x))^T +
\Ttau\grads\lambda_m(y)$. Looking at these, we find that they can be
calculated using singularity subtraction techniques discussed e.g. in
\cite{singsubshape}.

In the following we shall compare numerical values of
\begin{equation}
\label{equ:pdev_s_I}
\pdev{s} I_m^s\subst{s=0},\quad m = 1,2
\end{equation}
calculated with automatic differentiation applied to the singularity
subtraction technique, analytical formulas introduced here, and forward
difference formula. We inspect these values with four different
triangle configurations that occur in computations: the triangles
$T_p$ and $T_q$ do not touch, they share a vertex, they share an edge,
and the triangles are equal.
These four cases are presented in Figure \ref{fig:cases}.

The automatic differentiation algorithm generates code which computes
the derivatives of \eqref{equ:pdev_s_I}. The non-differentiated code
calculates the system matrix contributions of two triangles with
singularity subtraction technique by removing the two most singular
terms from the kernel of the EFIE integral. To get directly comparable results, we
also subtract two terms when computing the derivatives using
analytical formulas. For details and an example where the AD technique
is employed in EFIE calculations we refer to \cite{Toivanen2009ESA}.

We compute the derivatives of the local system matrix contributions
\begin{align}
  a_{pq}^s(\uu,\vv)= & 
  \frac{i}{\omega\varepsilon}I_2(\piola{F_s}\uu,\piola{F_s}\vv;F_s(T_p),F_s(T_q)) - \nonumber\\
  &i\omega\mu I_1(\piola{F_s}\uu,\piola{F_s}\vv;F_s(T_p),F_s(T_q)),
  \label{equ:diff_expr}
\end{align}
where $\uu$ and $\vv$ vary over three local $\RT$ basis functions on
$T_p$ and $T_q$ respectively, and investigate the effect of changing the
number of integration points $n$. The quadrature rules were adapted
from \cite{dunavant85}.

In Tables \ref{tab:near_imag}--\ref{tab:same_imag} we inspect the sum
\begin{equation}
\label{equ:sum_expr}
\Im \sum_{\uu,\vv} \pdev{s}a_{pq}^s(\uu,\vv),
\end{equation}
computed with the AD method, analytical formulas and the FD method given by
\begin{equation}
f'(s) \approx (f(s+h)-f(s))/h.
\end{equation}
The step length parameter was chosen to be $h=10^{-8}$. The reason for
presenting only the imaginary part is that it is the most involved
integral to compute.

In Tables \ref{tab:near_rel}--\ref{tab:same_rel} we have the difference in
the Frobenius norm of the $3\times 3$ matrices $\pdev{s} a_{pq}^s$
relative to the values computed using the present methods.

We see that the proposed analytical formulas and the automatic differentiation 
approach produce same results up to numerical precision. 
Relative difference of the derivatives obtained with the FD
method is of the order $10^{-7}$. This is expected, since
the forward finite difference formula has a truncation error
of the order $\mathcal{O}(h)$.

\section{Concluding remarks}
We have derived compact analytical formulas to calculate the shape derivatives
of the EFIE system matrix arising from RT discretization.
The derivatives are easy
to implement, since the emerging integrals can
be calculated using similar algorithms as the original EFIE system
matrix. In this case we used singularity subtraction technique, but 
the use of the singularity cancellation methods \cite{singcanc}
could be possible as well.

We have shown that the shape derivatives are
the same, up to numerical precision, as the ones computed using
automatic differentiation (AD), if one uses the same techniques
to evaluate the arising integrals. On the other hand,
AD is known to produce the exact derivatives of the given computer
realization. Such derivatives are perfectly consistent with the
objective function values, which is very convenient from the
optimization perspective.
The accuracy of the derivatives is naturally dictated
by the utilized numerical methods, and depends heavily for example
on the number of points used for the numerical integration.

The analytical formulas offer several benefits over the
automatic differentiation approach.
Existing AD tools are available only for some programming
languages, and while using these tools, one may be forced to restructure
the code and avoid some features of the programming language. 
Moreover, AD has to be applied also to all
subroutines that are called from the code, which may present
difficulties if one wishes to use any external libraries.
Finally, implementing the analytical formula is likely to 
lead to better computational performance, as one avoids execution
overhead related to some AD techniques, and the programmer is free to
optimize the code.

The FD derivatives are approximate by nature,
and their accuracy greatly depends on the step length parameter.
From the numerical results we can conclude that
the derivatives computed with the proposed formulas are in agreement with
the finite difference derivatives.
Use of the analytical formulas is therefore preferred, because 
the need to choose a step length parameter is avoided.

By having an analytical formula at hand, one is free to use any
suitable method for its numerical evaluation. Thus it is easy to make
a trade-off between accuracy and computational efficiency.  Moreover,
the analytical formulas provide a basis for understanding the
behaviour of the derivatives.

\begin{figure}
  \begin{center}
    \subfloat[\texttt{near}]{\includegraphics{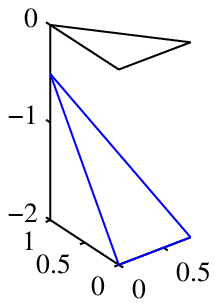}}
    \subfloat[\texttt{point}]{\includegraphics{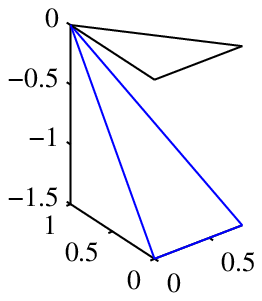}}
    \subfloat[\texttt{edge}]{\includegraphics{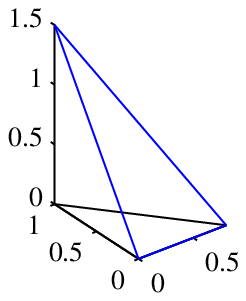}}
    \subfloat[\texttt{same}]{\includegraphics{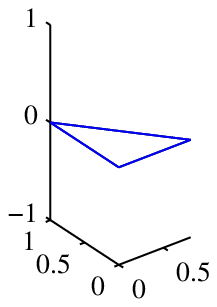}}
    \caption{Three different configurations of triangles $T_p$ and
      $T_q$ triangles. The triangle $T_q$ remains still in each of the
      cases.}
    \label{fig:cases}
    \end{center}
\end{figure}

\begin{table}
\centering
\caption{The imaginary part of $I(\uu,\vv)$ defined by the equation \eqref{equ:sum_expr}}
\subfloat[\texttt{near}\label{tab:near_imag}]{
\begin{tabular}{c|c|c|c}
  $n$ & Analytical & AD & FD \\
\hline
\input{output1_near.txt}
\end{tabular}}\hspace{0.1cm}
\subfloat[\texttt{point}]{\label{tab:point_imag}
\begin{tabular}{c|c|c|c}
  $n$ & Analytical & AD & FD \\
\hline
\input{output1_point.txt}
\end{tabular}}\\
\subfloat[\texttt{edge}]{\label{tab:edge_imag}
\begin{tabular}{c|c|c|c}
  $n$ & Analytical & AD & FD \\
\hline
\input{output1_edge.txt}
\end{tabular}}\hspace{0.1cm}
\subfloat[\texttt{same}]{\label{tab:same_imag}
\begin{tabular}{c|c|c|c}
  $n$ & Analytical & AD & FD \\
\hline
\input{output1_same.txt}
\end{tabular}}
\end{table}

\begin{table}
\centering
\caption{Relative difference to analytical derivative}
\subfloat[\texttt{near}\label{tab:near_rel}]{
\begin{tabular}{c|c|c}
  $n$ & AD & FD \\
\hline
\input{output2_near.txt}
\end{tabular}}\hspace{0.1cm}
\subfloat[\texttt{point}]{\label{tab:point_rel}
\begin{tabular}{c|c|c}
  $n$ & AD & FD \\
\hline
\input{output2_point.txt}
\end{tabular}}\\
\subfloat[\texttt{edge}]{\label{tab:edge_rel}
\begin{tabular}{c|c|c}
  $n$ & AD & FD \\
\hline
\input{output2_edge.txt}
\end{tabular}}\hspace{0.1cm}
\subfloat[\texttt{same}]{\label{tab:same_rel}
\begin{tabular}{c|c|c}
  $n$ & AD & FD \\
\hline
\input{output2_same.txt}
\end{tabular}}

\end{table}

\pagebreak

\bibliographystyle{plain}
\bibliography{defie_anal}

\end{document}

%% file: output1_near.txt
\texttt{6} & \texttt{2.98169473e+01} & \texttt{2.98169473e+01} & \texttt{2.98169583e+01} \\ 
\texttt{16} & \texttt{2.98097099e+01} & \texttt{2.98097099e+01} & \texttt{2.98097110e+01} \\ 
\texttt{61} & \texttt{2.98096873e+01} & \texttt{2.98096873e+01} & \texttt{2.98096885e+01} \\ 
\texttt{85} & \texttt{2.98096873e+01} & \texttt{2.98096873e+01} & \texttt{2.98096876e+01} \\ 

%% file: output1_point.txt
\texttt{6} & \texttt{-6.13463371e+00} & \texttt{-6.13463371e+00} & \texttt{-6.13463194e+00} \\ 
\texttt{16} & \texttt{-6.11420678e+00} & \texttt{-6.11420678e+00} & \texttt{-6.11420794e+00} \\ 
\texttt{61} & \texttt{-6.11421020e+00} & \texttt{-6.11421020e+00} & \texttt{-6.11421161e+00} \\ 
\texttt{85} & \texttt{-6.11416794e+00} & \texttt{-6.11416794e+00} & \texttt{-6.11417224e+00} \\ 

%% file: output1_edge.txt
\texttt{6} & \texttt{-7.84966165e+01} & \texttt{-7.84966165e+01} & \texttt{-7.84966142e+01} \\ 
\texttt{16} & \texttt{-7.67762767e+01} & \texttt{-7.67762767e+01} & \texttt{-7.67762887e+01} \\ 
\texttt{61} & \texttt{-7.68970360e+01} & \texttt{-7.68970360e+01} & \texttt{-7.68970451e+01} \\ 
\texttt{85} & \texttt{-7.68875833e+01} & \texttt{-7.68875833e+01} & \texttt{-7.68875768e+01} \\ 

%% file: output1_same.txt
\texttt{6} & \texttt{-5.98000685e+01} & \texttt{-5.98000685e+01} & \texttt{-5.98000312e+01} \\ 
\texttt{16} & \texttt{-5.86918000e+01} & \texttt{-5.86918000e+01} & \texttt{-5.86917977e+01} \\ 
\texttt{61} & \texttt{-5.87792043e+01} & \texttt{-5.87792043e+01} & \texttt{-5.87792522e+01} \\ 
\texttt{85} & \texttt{-5.87738744e+01} & \texttt{-5.87738744e+01} & \texttt{-5.87738958e+01} \\ 

%% file: output2_near.txt
\texttt{6} & \texttt{1.10e-15} & \texttt{2.14e-07} \\ 
\texttt{16} & \texttt{1.11e-15} & \texttt{5.27e-08} \\ 
\texttt{61} & \texttt{1.85e-15} & \texttt{3.02e-08} \\ 
\texttt{85} & \texttt{2.52e-15} & \texttt{2.74e-08} \\ 

%% file: output2_point.txt
\texttt{6} & \texttt{1.07e-15} & \texttt{1.30e-07} \\ 
\texttt{16} & \texttt{7.25e-16} & \texttt{9.42e-08} \\ 
\texttt{61} & \texttt{2.05e-15} & \texttt{4.73e-08} \\ 
\texttt{85} & \texttt{1.24e-15} & \texttt{1.57e-07} \\ 

%% file: output2_edge.txt
\texttt{6} & \texttt{5.53e-16} & \texttt{3.51e-08} \\ 
\texttt{16} & \texttt{3.21e-15} & \texttt{1.16e-07} \\ 
\texttt{61} & \texttt{7.59e-15} & \texttt{9.78e-08} \\ 
\texttt{85} & \texttt{7.82e-15} & \texttt{6.34e-08} \\ 

%% file: output2_same.txt
\texttt{6} & \texttt{6.13e-15} & \texttt{4.17e-07} \\ 
\texttt{16} & \texttt{1.84e-15} & \texttt{9.87e-08} \\ 
\texttt{61} & \texttt{3.06e-15} & \texttt{5.39e-07} \\ 
\texttt{85} & \texttt{1.08e-15} & \texttt{2.44e-07} \\ 

%% file: defie_anal.bbl
\begin{thebibliography}{10}

\bibitem{schwab_bem}
A.~Buffa, M.~Costabel, and C.~Schwab.
\newblock Boundary element methods for {Maxwell's} equations on non-smooth
  domains.
\newblock {\em Numerische Mathematik}, 92:679--710, 2002.

\bibitem{buffa_bem}
Annalisa Buffa and Ralf Hiptmair.
\newblock Galerkin boundary element methods for electromagnetic scattering.
\newblock In {\em Topics in computational wave propagation}, volume~31 of {\em
  Lect. Notes Comput. Sci. Eng.}, pages 83--124. Springer, Berlin, 2003.

\bibitem{efie_christiansen}
Snorre~H. Christiansen.
\newblock Discrete {Fredholm} properties and convergence estimates for the
  electric field integral equation.
\newblock {\em Math. Comput.}, 73:143--167, January 2004.

\bibitem{costabel_shape_deriv}
Martin Costabel and Fr\'ed'erique~Le Lou\"er.
\newblock Shape derivatives of boundary integral operators in electromagnetic
  scattering.
\newblock {\em arXiv}, 2010.

\bibitem{costabel_shape_deriv_II}
Martin Costabel and Fr\'ed'erique~Le Lou\"er.
\newblock Shape derivatives of boundary integral operators in electromagnetic
  scattering. {Part II}: Application to scattering by a homogeneous dielectric
  obstacle.
\newblock {\em arXiv}, 2011.

\bibitem{dunavant85}
D.~A. Dunavant.
\newblock High degree efficient symmetrical {Gaussian} quadrature rules for the
  triangle.
\newblock {\em International Journal for Numerical Methods in Engineering},
  21(6):1129--1148, 1985.

\bibitem{singsubshape}
Seppo J\"arvenp\"a\"a, Matti Taskinen, and Pasi Yl\"a-Oijala.
\newblock Singularity subtraction technique for high-order polynomial vector
  basis functions on planar triangles.
\newblock {\em IEEE Transactions on Antennas and Propagation}, 54(1):42--49,
  January 2006.

\bibitem{kolundzija-djordjevic}
B.M. Kolund{\v{z}}ija and A.R. Djordjevi{\'c}.
\newblock {\em Electromagnetic modeling of composite metallic and dielectric
  structures}.
\newblock Artech House Electromagnetic Analysis Series. Artech House, 2002.

\bibitem{nikolova2006_avm}
N.K. Nikolova, Jiang Zhu, Dongying Li, M.H. Bakr, and J.W. Bandler.
\newblock Sensitivity analysis of network parameters with electromagnetic
  frequency-domain simulators.
\newblock {\em IEEE Transactions on Microwave Theory and Technique}, 54(2):670
  -- 681, February 2006.

\bibitem{efie}
A.J. Poggio and E.K. Miller.
\newblock {\em Computer techniques for electromagnetics}, chapter Integral
  equation solutions of three-dimensional scattering problems.
\newblock Oxford, UK: Pergamon Press, 1973.

\bibitem{potthast_dom_deriv}
Roland Potthast.
\newblock Domain derivatives in electromagnetic scattering.
\newblock {\em Math. Methods Appl. Sci.}, 19(15):1157--1175, 1996.

\bibitem{rwg}
Sadasiva~M. Rao, Donald~R. Wilton, and Allen~W. Glisson.
\newblock Electromagnetic scattering by surfaces of arbitrary shape.
\newblock {\em IEEE Transactions on Antennas and Propagation}, 30(3):409--418,
  May 1982.

\bibitem{raviart-thomas}
P.-A. Raviart and J.~M. Thomas.
\newblock A mixed finite element method for 2nd order elliptic problems.
\newblock In {\em Mathematical aspects of finite element methods (Proc. Conf.,
  Consiglio Naz. delle Ricerche (C.N.R.), Rome, 1975)}, pages 292--315. Lecture
  Notes in Math., Vol. 606. Springer, Berlin, 1977.

\bibitem{rumsey54}
V.~H. Rumsey.
\newblock {Reaction} {Concept} in {Electromagnetic} {Theory}.
\newblock {\em Phys. Rev.}, 94:1483--1491, June 1954.

\bibitem{Toivanen2009ESA}
J.~I. Toivanen, R.~A.~E. M{\"a}kinen, S.~J{\"a}rvenp{\"a}{\"a},
  P.~Yl{\"a}-Oijala, and J.~Rahola.
\newblock Electromagnetic sensitivity analysis and shape optimization using
  method of moments and automatic differentiation.
\newblock {\em IEEE Transactions on Antennas Propagation}, 57(1):168--175,
  2009.

\bibitem{uwb_splines}
Jukka~I. Toivanen, Raino A.~E. M{\"a}kinen, Jussi Rahola, Seppo
  J{\"a}rvenp{\"a\"a}, and Pasi Yl{\"a}-Oijala.
\newblock Gradient-based shape optimisation of ultra-wideband antennas
  parameterised using splines.
\newblock {\em IET Microwaves, Antennas and Propagation}, 4:1406--1414, 2010.

\bibitem{singcanc}
D.~Wilton, S.~Rao, A.~Glisson, D.~Schaubert, O.~Al-Bundak, and C.~Butler.
\newblock Potential integrals for uniform and linear source distributions on
  polygonal and polyhedral domains.
\newblock {\em IEEE Transactions on Antennas and Propagation}, 32(3):276 --
  281, March 1984.

\end{thebibliography}
